\documentclass[a4paper,reqno]{article}
\usepackage{amsmath,amsfonts,amssymb,amsthm,amsrefs,bbm,graphicx,enumerate,geometry}
\usepackage{mathrsfs}

\usepackage{subfigure} 

\usepackage{color}

\usepackage{framed} 
\usepackage{hyperref}  
\hypersetup{
    linktoc=page,
    linkcolor=red,          
    citecolor=blue,        
    filecolor=blue,      
    urlcolor=cyan,
   colorlinks=true           
}

\numberwithin{equation}{section}

\newtheorem{thm}{Theorem}[section]

\newtheorem{cor}[thm]{Corollary}
\newtheorem{defn}[thm]{Definition}

\newtheorem{remark}[thm]{Remark}

\renewcommand{\epsilon}{\varepsilon}

\def\<#1{\langle #1\rangle}

\begin{document}{\allowdisplaybreaks[4]}


\title{Multiple chordal SLE($\kappa$) and quantum Calogero-Moser system}
\author{
    Jiaxin Zhang
    \footnotemark[1]
   }
\renewcommand{\thefootnote}{\fnsymbol{footnote}}

\footnotetext[1]{{\bf zhangjx@caltech.edu} Department of Mathematics, California Institute of Technology}

\maketitle

We study multiple chordal SLE$(\kappa)$ systems in a simply connected domain $\Omega$, where $z_1, \ldots, z_n \in \partial \Omega$ are boundary starting points and $q \in \partial \Omega$ is an additional marked boundary point.

As a consequence of the domain Markov property and conformal invariance, we show that the presence of the marked boundary point $q$ gives rise to a natural equivalence relation on partition functions. While these functions are not necessarily conformally covariant, each equivalence class contains a conformally covariant representative.

Building on the framework introduced in \cite{Dub07}, we demonstrate that in the $\mathbb{H}$-uniformization with $q = \infty$, the partition functions satisfy both the null vector equations and a dilatation equation with scaling exponent $d$.

Using techniques from the Coulomb gas formalism in conformal field theory, we construct two distinct families of solutions, each indexed by a topological link pattern of type $(n, m)$ with $2m \leq n$.

In the special case $\Omega = \mathbb{H}$ and $q = \infty$, we further show that these partition functions correspond to eigenstates of the quantum Calogero–Moser system, thereby extending the known correspondence beyond the standard $(2n, n)$ setting.
\tableofcontents

\newpage

\section{Introduction}
\subsection{Background}
\
\indent 
The Schramm--Loewner evolution $\mathrm{SLE}(\kappa)$, introduced in \cites{Sch00, LSW04}, is a family of conformally invariant random curves with a positive real parameter $\kappa > 0$. It arises as the scaling limit of interfaces in two-dimensional critical lattice models.

A parallel approach to understanding critical phenomena is provided by conformal field theory (CFT), which studies local observables through correlation functions \cites{Car96, FK04}. The rigorous interplay between SLE and CFT---often referred to as the SLE–CFT correspondence has been studied in \cites{BB03a, FW03, FK04, Dub15a, Pel19}. Under this correspondence, the central charge $c(\kappa)$ associated with $\mathrm{SLE}(\kappa)$ is given by
\[
c(\kappa) = \frac{(3\kappa - 8)(6 - \kappa)}{2\kappa}.
\]

Multiple SLE systems extend the single-curve construction to a collection of $n$ interacting random curves. In the standard chordal setup, $2n$ marked boundary points are connected in pairs by $n$ disjoint SLE-type paths, forming a configuration of type $(2n,n)$ \cites{Dub06, KL07, Law09b, FK15a, PW19, PW20}. These systems are encoded by partition functions satisfying a system of second-order partial differential equations known as the null vector equations.

In this paper, we extend the theory of standard multiple chordal SLE$(\kappa)$ (type $(2n,n)$) systems to general configurations of type $(n,m)$. A type $(n,m)$ configuration represents a topological pattern characterized by $n$ boundary points, one additional marked boundary point $u$, and $M$ non-intersecting connections (or pairings) among the $n$ boundary points. The remaining boundary points are connected to the marked point $u$.

\subsection{Multiple chordal SLE($\kappa$) systems with $\kappa>0$}

In a simply connected domain $\Omega$ with boundary points $z_1, z_2, \ldots, z_n$ and a marked interior point $q$, we define a \emph{local multiple chordal SLE($\kappa$) system} as a compatible family of probability measures  
\[
\mathbb{P}_{\left(\Omega; z_1, z_2, \ldots, z_n, u\right)}^{\left(U_1, U_2, \ldots, U_n\right)}
\]  
on $n$-tuples of continuous, non-self-crossing curves starting from $z_i$ within a localization neighborhood $U_i$, none of which contains $u$. A more precise characterization of these measures is provided in Definitions \ref{localization of measure} and \label{local multiple chordal SLE(kappa)}.
 In the upper half plane $\mathbb{H}$ with
$n$ marked boundary points $\{z_1,z_2,\ldots,z_n\}$ and one additional marked point $u =\infty$, we use the notation $\mathbb{P}_{(x_1,\ldots,x_n)}$ where $x_1 <x_2<\ldots<x_n$.

\begin{defn}[Localization of measures]\label{localization of measure}
Let $\Omega \subsetneq \mathbb{C}$ be a simply connected domain, and let $u$ be a marked boundary point. Let $z_1, z_2, \ldots, z_n$ be distinct prime ends of $\partial \Omega$, and let $U_1, U_2, \ldots, U_n$ be disjoint closed neighborhoods of $z_1, z_2, \ldots, z_n$ in $\Omega$ such that $u \notin U_j$ for all $j$ and $U_i \cap U_j = \emptyset$ for $1 \leq i < j \leq n$. 

We consider probability measures
\[
\mathbb{P}_{\left(\Omega; z_1, z_2, \ldots, z_n, u\right)}^{\left(U_1, U_2, \ldots, U_n\right)}
\]
on $n$-tuples of unparametrized continuous curves, where each curve starts at $z_j$, is contained in $U_j$, and exits $U_j$ almost surely. 

A family of such measures, indexed by different choices of neighborhoods $\left(U_1, U_2, \ldots, U_n\right)$, is called \emph{compatible} if for all $U_j \subset U_j'$, the measure $\mathbb{P}_{\left(\Omega; z_1, z_2, \ldots, z_n, u\right)}^{\left(U_1, U_2, \ldots, U_n\right)}$ is obtained from 
\[
\mathbb{P}_{\left(\Omega; z_1, z_2, \ldots, z_n, u\right)}^{\left(U_1', U_2', \ldots, U_n'\right)}
\]
by restricting each curve to the portion before it first exits the smaller domain $U_j$.
\end{defn}

\begin{figure}[h]
    \centering
    \includegraphics[width=8cm]{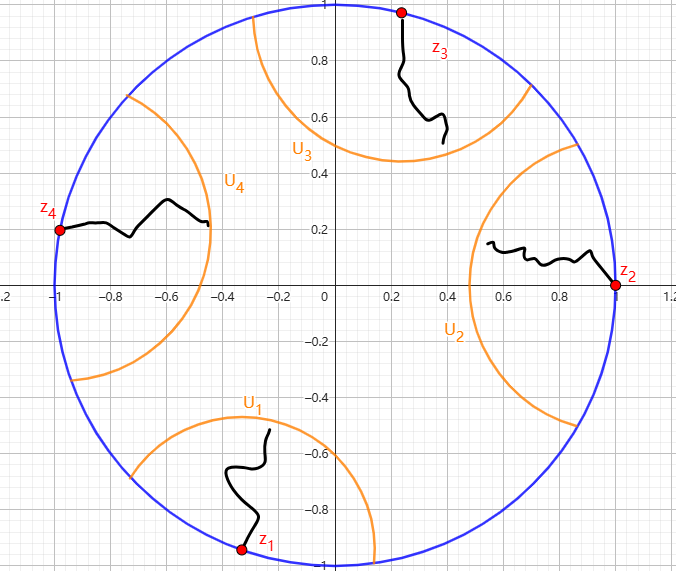 }
    \caption{Localization of multiple chordal SLE($\kappa$)} 
   
\end{figure}

\begin{defn}[Local multiple chordal SLE($\kappa$)]
The locally commuting n chordal $\mathrm{SLE}_\kappa$ is a compatible family of measures $\mathbb{P}_{\left(\Omega; z_1, z_2,\ldots,z_n,u\right)}^{\left(U_1, U_2,\ldots,U_n\right)}$ on $n$ tuples of continuous non-self-crossing curves $\left(\eta^{(1)}, \eta^{(2)},\ldots,\eta^{(n)}\right)$ for all simply connected domain $\Omega$ with $\left(z_1, z_2,\ldots,z_n,u\right)$, and $\left(U_1, U_2,\ldots,U_n\right)$ as above that satisfy additionally the conditions below. 

\begin{itemize}
    \item \textbf{Conformal invariance}: If $\varphi: \Omega \rightarrow \Omega^{\prime}$ is a conformal map fixing 0, then the pullback measure
$$
\varphi^* \mathbb{P}_{\left(\Omega^{\prime} ; \varphi\left(z_1\right), \varphi\left(z_2\right),\ldots,\varphi\left(z_n),\varphi(u)\right)\right)}^{\left(\varphi\left(U_1\right), \varphi\left(U_2\right),\right)}=\left(\mathbb{P}_{\left(\Omega ; z_1, z_2,\ldots,z_n,u\right)}^{\left(U_1, U_2,\ldots,U_n\right)}\right) .
$$

Therefore, it suffices to describe the measure when $\left(\Omega ; z_1, z_2,\ldots,z_n,u\right)=\left(\mathbb{H}; z_1,z_2,\ldots,z_n,0\right)$.
We can extend our definition to arbitrary simply-connected domain $\Omega$ with a marked interior point $u$ by pulling back via a conformal equivalence $\phi:\Omega \rightarrow \mathbb{D}$ sending $u$ to $0$.

\item
\textbf{Domain Markov property}: Let $\left(\gamma^{(1)}, \gamma^{(2)},\ldots,\gamma^{(n)}\right) \sim \mathbb{P}_{\left(\mathbb{H} ; x_1, x_2,\ldots,x_n\right)}^{\left(U_1, U_2,\ldots, U_n\right)}$ and we parametrize $\gamma^{(j)}$ by their own capacity in $\mathbb{H}$. Let $\mathbf{t}=\left(t_1, t_2,\ldots,t_n\right)$, such that $t_j$ is a stopping time for $\gamma^{(j)}$ and $\gamma_{\left[0, t_j\right]}^{(j)}$ is contained in the interior of $U_j$. Let
$$
\tilde{U}_j=U_j \backslash \gamma_{\left[0, t_j\right]}^{(j)}, \quad \tilde{\gamma}^{(j)}=\gamma^{(j)} \backslash \gamma_{\left[0, t_j\right]}^{(j)}, \quad j=1,2,\ldots,n; \quad \tilde{\Omega}=\mathbb{H} \backslash\left(\gamma_{\left[0, t_1\right]}^{(1)} \cup \gamma_{\left[0, t_2\right]}^{(2)} \cup \ldots \cup \gamma_{\left[0, t_n\right]}^{(n)}\right) .
$$

Then conditionally on $\gamma_{\left[0, t_1\right]}^{(1)} \cup \gamma_{\left[0, t_2\right]}^{(2)} \cup \ldots \cup  \gamma_{\left[0, t_n\right]}^{(n)}$, we have
$$
\left(\tilde{\gamma}^{(1)}, \tilde{\gamma}^{(2)}, \ldots, \tilde{\gamma}^{(n)}\right) \sim \mathbb{P}_{\left(\tilde{\Omega} ; \gamma_{t_1}^{(1)}, \gamma_{t_2}^{(2)},\ldots,\gamma_{t_n}^{(n)}\right)}^{\left(\tilde{U}_1, \tilde{U}_2, \ldots,\tilde{U}_n\right)} .
$$

\item \textbf{Absolute continuity with respect to independent SLE($\kappa$)}:
 In $\mathbb{H}$-uniformization, let $\left(\gamma^{(1)}, \gamma^{(2)},\ldots,\gamma^{(n)}\right) \sim \mathbb{P}_{\left(\mathbb{H} ; x_1,x_2,\ldots,x_n,u\right)}^{\left(U_1, U_2,\ldots,U_n)\right.}$. 

We assume that there exist smooth functions $b_j: \mathfrak{X}^n(\boldsymbol{x},u) \rightarrow \mathbb{R}$, where the chamber: $$\mathfrak{X}^n(\boldsymbol{x},u)=\left\{(x_1,x_2,\ldots,x_n,u) \in \mathbb{R}^n \times \mathbb{H} \mid x_1<x_2<\ldots<x_n, u \in \mathbb{H} \right\}$$ 
such that the capacity parametrized Loewner driving function $t \mapsto \theta_t^{(j)}$ of $\gamma^{(j)}$ satisfies
\begin{equation} \label{H marginal}
\left\{\begin{array}{l}
\mathrm{~d} x_{j}(t)=\sqrt{\kappa} \mathrm{d} B_{j}(t)+b_j\left(x_{1}(t),x_{2}(t),\ldots,x_{n}(t),u\right) \mathrm{d} t \\
\mathrm{~d} x_{k}(t)=\frac{2}{x_{k}(t)-x_{j}(t)} \mathrm{d} t,  k\neq j
\end{array}\right.
\end{equation}
where $B_{j}$ is one-dimensional standard Brownian motion. 

\end{itemize}
\end{defn}

In particular, the domain Markov property implies that we can first map out $\gamma_{\left[0, t_i\right]}^{(i)}$ using $g_{t_i}^{(i)}$, then mapping out $g_{t_i}^{(i)}\left(\gamma_{\left[0, t_j\right]}^{(j)}\right)$, or vice versa. The image has the same law regardless of the order in which we map out the curves. This is also known as the commutation relations or reparametrization symmetry.

The core principle in our paper is the SLE-CFT correspondence. SLE and multiple SLE systems can be coupled to a conformal field in two key aspects:
\begin{itemize} 
\item The level-two degeneracy equations for the conformal fields coincide with the null vector equations for the SLE partition functions. 
\item The correlation functions of the conformal fields serve as martingale observables for the SLE processes. \end{itemize}

We make initial progress in studying general type $(n,m)$ multiple chordal SLE($\kappa$) systems by exploring the following aspects:
\begin{itemize}
    \item Commutation relations and conformal invariance
    \item Solution space of the null vector equations.
\end{itemize}

Here, commutation relations for multiple SLE($\kappa$) mean that we can grow the multiple SLE($\kappa$) curves in arbitrary order and yield the same result.

Extending the results in \cite{Dub07} on commutation relations, we derive analogous commutation relations for multiple chordal SLEs in the upper half plane $\mathbb{H}$ with
$n$ starting $x_1,x_2,\ldots,x_n \in \partial{\mathbb{H}}$ and one additional marked point $u =\infty$, see section \ref{reparametrization symmetry}. 

The family of measure $\mathbb{P}_{(x_1,\ldots,x_n)}$ of a general multiple chordal SLE($\kappa$) system is encoded by a partition function 
$\psi(\boldsymbol{x}):\left\{\left(x_1, x_2,\ldots,x_n\right) \in \mathbb{R}^n \mid x_1<x_2<\ldots<x_n\right\} \rightarrow \mathbb{R}_{>0}$, see section (\ref{reparametrization symmetry}) for detailed explanation.

\begin{thm}
For a local multiple chordal SLE(\(\kappa\)) system, there exists a positive function \( \psi(\boldsymbol{x}) \) such that the drift term \( b_j \) in the marginal laws satisfies:
\begin{equation}
    b_j = \kappa \frac{\partial_j \psi}{\psi}, \quad j = 1,2,\dots,n.
\end{equation}  

As a consequence of the commutation relations, \( \psi \) satisfies the null vector equation  
\begin{equation} \label{null vector equations in mathbb H infty 0 constant}
\frac{\kappa}{2} \, \partial_{ii} \psi + \sum_{j \neq i} \frac{2}{x_j - x_i} \, \partial_i \psi 
+ \left(1 - \frac{6}{\kappa} \right) \sum_{j \neq i} \frac{\psi}{(x_j - x_i)^2} = 0,
\end{equation}

Moreover, by conformal (rotational) invariance, for any conformal map \( \tau \in {\rm Aut}(\mathbb{H},\infty) \) of the form \( \tau(z) = az + b \), the drift term \( b_i(\boldsymbol{x},u) \) transforms as a pre-Schwarzian form:  
\begin{equation}
    b_i = \tau^{\prime} \widetilde{b}_i \circ \tau + \frac{6-\kappa}{2} \left(\log \tau^{\prime}\right)^{\prime} = a \cdot \widetilde{b}_i \circ \tau.
\end{equation}  

In particular, there exists a dilation constant \( d \) such that  
\begin{equation} \label{conformal invariant Aut H infty}
\psi(x_1, x_2, \dots, x_n, \infty) 
= a^{\frac{n(\kappa-6)}{2\kappa} - d} \psi(ax_1+b, ax_2+b, \dots, ax_n+b, \infty).
\end{equation}  
\end{thm}

A significant difference between the general $(n,m)$ multiple chordal SLE($\kappa$) systems and $(2n,n)$ type multiple chordal SLE($\kappa$) systems arises when we study their conformal invariance properties. Although the multiple chordal SLE($\kappa$) systems are conformally invariant, the partition functions in its corresponding equivalence classes do not necessarily exhibit conformal covariance when we have an extra marked point.

We define two partition functions as \emph{equivalent} if and only if they induce identical multiple chordal SLE($\kappa$) systems. Equivalent partition functions differ a by multiplicative function $f(u)$.
\begin{equation}
\tilde{\psi}=f(u)\cdot\psi
\end{equation}
where $f(u)$ is an arbitrary positive real smooth function depending on the marked boundary point $u$
A simple example that violates conformal covariance is
when $f(u)$ is not conformally covariant. 
However, within each equivalence class, it is still possible to find at least one conformally covariant partition function.

Following \cite{FK15c} on solution space of the null vector equations for partition functions of multiple chordal SLE($\kappa$), we construct two types of solutions to the null vector equations and Ward's identities for partition functions of multiple radial SLE($\kappa$) via screening method in conformal field theory, see section \ref{Classification of screening solutions}. 

\begin{thm}
\label{solution space to null and ward}
The following two types of Coulomb gas integrals (see definitions in Section~\ref{Classification of screening solutions}) solve the null vector equation \eqref{null vector equations in mathbb H infty 0 constant} and the dilatation equation \eqref{conformal invariant Aut H infty}:
\begin{itemize}
    \item[(1)] For any link pattern \( \alpha \in LP(n,m) \), with integers \( m,n \) satisfying \( 1 \leq m \leq \frac{n}{2} \), the corresponding Coulomb gas integral \( \mathcal{J}_{\alpha}^{n,m}(\boldsymbol{\boldsymbol{x}}) \) solves the null vector equations ~\eqref{null vector equations in mathbb H infty 0 constant} and dilatation equation ~\eqref{conformal invariant Aut H infty} with 
        \[
    d = \frac{n(\kappa - 6)}{2\kappa} - \lambda_{(b)}(u)=\frac{n(\kappa - 6)}{2\kappa}
    - \frac{4m^2 - 4mn + n^2 - 4m + 2n}{\kappa}
    - \frac{2m - n}{2},
    \]
    where
    \[
    \lambda_{(b)}(u) = \frac{(2m - n)^2}{\kappa} - \frac{2(2m - n)}{\kappa} + \frac{2m - n}{2}.
    \]
    is the conformal dimension at $u$.

    \item[(2)] For any link pattern \( \alpha \in LP(n,m) \), with integers $m,n$ and \( 1 \leq m \leq \frac{n}{2} \), the corresponding Coulomb gas integral \( \mathcal{K}_{\alpha}^{n,m}(\boldsymbol{\boldsymbol{x}}) \) solves the null vector equations ~\eqref{null vector equations in mathbb H infty 0 constant} and dilatation equation ~\eqref{conformal invariant Aut H infty} with 
        \[
    d= \frac{n(\kappa - 6)}{2\kappa} - \lambda_{(b)}(u)=\frac{n(\kappa - 6)}{2\kappa}
    - \frac{4m^2 - 4mn + n^2 - 4m + 2n}{\kappa}
    + \frac{2m - n - 2}{2},
    \]
    where
    \[
    \lambda_{(b)}(u) = \frac{(2m - n)(2m - n - 2)}{\kappa} - \frac{2m - n - 2}{2}.
    \]
    is the conformal dimension at $u$.
\end{itemize}
\end{thm}

Here, \( \alpha \) denotes the integration contour, and \( LP(n,m) \) represents the set of all possible multiple integration contours with \( n \) boundary points and \( m \) integration variables. The abbreviation \( LP \) stands for link pattern, which is defined in Section~\ref{Classification of screening solutions}.

A comprehensive analysis of the linear independence among the constructed screening solutions is left for future investigation. The problem of characterizing the full solution space to the null vector equations, particularly in the chordal setting with a marked boundary point, remains open.

The classification of multiple chordal $\mathrm{SLE}(\kappa)$ systems is connected to the study of positive solutions to the null vector equations and Ward identities. These positive solutions determine the admissible partition functions and thus correspond to the space of conformally invariant probability measures and commutation properties of the SLE curves.

\subsection{Relations to quantum Calogero-Moser systems}

We show that a partition function satisfying the null vector equations (\ref{null vector equations in mathbb H infty 0 constant}) corresponds to an eigenfunction of the quantum Calogero-Sutherland Hamiltonian, as first discovered in \cite{Car04}.

\begin{thm}\label{CM results kappa>0}
The multiple chordal SLE($\kappa$) is described by the partition function $\mathcal{Z}(\boldsymbol{x})$, we have
\begin{equation}
\mathcal{L}_j \mathcal{Z}(\boldsymbol{x})= h \mathcal{Z}(\boldsymbol{x})   
\end{equation}
where $\mathcal{L}_j$ is the null vector differential operator:
\begin{equation}
\mathcal{L}_j=  \frac{\kappa}{2}\left(\frac{\partial}{\partial x_j}\right)^2 +\sum_{k \neq j}\left( \left(\frac{2}{x_k-x_j}\right) \frac{\partial}{\partial x_k}-\frac{6-\kappa}{\kappa}\frac{1}{\left(x_k-x_j\right)^2}\right)  
\end{equation}

\begin{itemize}

\item[(i)] We transform the partition function $\mathcal{Z}(\boldsymbol{x})$ into an eigenfunction of quantum Calogero-Moser system by multiplying Coulomb gas correlation $\Phi_{\frac{1}{\kappa}}^{-1}(\boldsymbol{x})$ 

\begin{equation}
\tilde{\mathcal{Z}}(\boldsymbol{x})=\Phi_{\frac{1}{\kappa}}^{-1}(\boldsymbol{x})\mathcal{Z}(\boldsymbol{x})   
\end{equation}

where $$\Phi_{r}(\boldsymbol{x})=\prod_{1 \leq j<k \leq n}\left(x_j-x_k\right)^{-2r}$$

 Then $\tilde{\mathcal{Z}}(\boldsymbol{x})$ satisfies 
 $$\left(\Phi_{\frac{1}{\kappa}}^{-1} \cdot \mathcal{L}_j \cdot \Phi_{\frac{1}{\kappa}}\right)  \tilde{\mathcal{Z}}(\boldsymbol{x}) = h\tilde{\mathcal{Z}}(\boldsymbol{x})$$  where the differential operator $\Phi_{\frac{1}{\kappa}}^{-1} \cdot \mathcal{L}_j \cdot \Phi_{\frac{1}{\kappa}}$ is given by

\begin{equation}
\begin{aligned}
\Phi_{\frac{1}{\kappa}}^{-1} \cdot \mathcal{L}_j \cdot \Phi_{\frac{1}{\kappa}}= & \frac{\kappa}{2} \partial_j^2- F_j \partial_j+\frac{1}{2\kappa} F_j^2-\frac{1}{2} F_j^{\prime} -\sum_{k \neq j}\left(f_{j k}\left(\partial_k-\frac{1}{\kappa} F_k\right)-\frac{6-\kappa}{2\kappa} f_{j k}^{\prime}\right) \\
&=\frac{\kappa}{2} \partial_j^2- \sum_{k} f_{jk}(\partial_j+\partial_k)\\
&-\frac{1}{\kappa}\left[\sum_k \sum_{l \neq k} f_{j k} f_{j l}-2\sum_{k\neq j} f_{j k}^2\right]- F_{j}'
\end{aligned}   
\end{equation}

 The sum of the null vector differential operators $\mathcal{L}=\sum \mathcal{L}_j$ is given by
\begin{equation}
\Phi_{-\frac{1}{\kappa}} \cdot \mathcal{L} \cdot \Phi_{\frac{1}{\kappa}}=\kappa H_n\left(\frac{8}{\kappa}\right)
\end{equation}

where $H_{n}(\beta)$ with $\beta=\frac{8}{\kappa}$ is  the quantum Calogero-Moser Hamiltonian
$$H_n(\beta)=\sum_{j=1}^n \frac{1}{2} \frac{\partial^2}{\partial x_j^2}-\frac{\beta(\beta-2)}{16} \sum_{1 \leq j<k \leq n} \frac{1}{\sin ^2\left(\frac{x_j-x_k}{2}\right)}
$$
\item[(ii)]
The commutation relation between growing two SLEs can be expressed as
$$[\mathcal{L}_j,\mathcal{L}_k]= \frac{1}{\sin^2(\frac{x_j-x_k}{2})}(\mathcal{L}_k-\mathcal{L}_j)$$
\end{itemize}
then
$$[\mathcal{L}_j,\mathcal{L}_k]\mathcal{Z}(\boldsymbol{x})= \frac{1}{\sin^2(\frac{x_j-x_k}{2})}(\mathcal{L}_k-\mathcal{L}_j)\mathcal{Z}(\boldsymbol{x})=0$$
\end{thm}

Notably, these solutions to the null vector PDE system in section \ref{Classification of screening solutions} yield eigenstates of the Calogero-Moser system beyond the eigenstates built upon the fermionic ground states.

\newpage

\section{Conformal covariance of partition functions} \label{Commutation relations and conformal invariance}

\subsection{Transformation of Loewner flow under coordinate change}
\label{transformation of Loewner under coordinate change}

In this section we show that the Loewner chain of a curve, when viewed in a different coordinate chart, is a time reparametrization of the Loewner chain in the standard coordinate chart but with different initial conditions. This result serves as a preliminary step towards understanding the local commutation relations and the conformal invariance of multiple SLE($\kappa$) systems.

\begin{thm}[Loewner coordinate change in $\mathbb{H}$]
\label{Loewner coordinate change in H}

Let $\gamma = \gamma(t)$ be a continuous, non-self-crossing curve in the closed upper half-plane $\overline{\mathbb{H}}$, with $\gamma(0) = x \in \mathbb{R}$ and $\gamma((0, t]) \subset \mathbb{H}$. Assume that $\gamma$ is generated by the Loewner chain
\[
\partial_t g_t(z) = \frac{2}{g_t(z) - W_t}, \quad \dot{W}_t = b(W_t, g_t(z_1), \ldots, g_t(z_m)), \quad g_0(z) = z, \quad W_0 = x,
\]
for some function $b : \mathbb{R} \times \mathbb{C}^m \to \mathbb{R}$. To simplify notation, we write $\dot{W}_t = b(W_t)$, with the understanding that $b$ may depend implicitly on the evolution of marked points $g_t(z_j)$.

Let $\Psi : \mathcal{N} \to \mathbb{H}$ be a conformal map defined on a neighborhood $\mathcal{N}$ of $x$, such that $\gamma([0,T]) \subset \mathcal{N}$ for some $T > 0$, and $\Psi$ maps $\partial \mathcal{N} \cap \mathbb{R}$ into $\mathbb{R}$. Define the image curve $\widetilde{\gamma}(t) := \Psi \circ \gamma(t)$.

Let $\widetilde{g}_t$ be the conformal map from $\mathbb{H} \setminus \widetilde{\gamma}([0,t])$ onto $\mathbb{H}$ normalized at infinity by $\widetilde{g}_t(z) = z + o(1)$ as $z \to \infty$. Define
\[
\Psi_t := \widetilde{g}_t \circ \Psi \circ g_t^{-1}.
\]
Then the Loewner chain $\widetilde{g}_t$ satisfies
\[
\partial_t \widetilde{g}_t(z) = \frac{2 \Psi_t'(W_t)^2}{\widetilde{g}_t(z) - \widetilde{W}_t}, \quad \widetilde{g}_0(z) = z, \quad \widetilde{W}_0 = \Psi(x),
\]
where the new driving function is given by
\[
\widetilde{W}_t = \widetilde{g}_t \circ \Psi \circ \gamma(t) = \Psi_t(W_t).
\]

The image curve $\widetilde{\gamma}$ is parameterized such that its half-plane capacity satisfies
\[
\operatorname{hcap}(\widetilde{\gamma}[0, t]) = 2 \sigma(t), \quad \text{where} \quad \sigma(t) := \int_0^t \Psi_s'(W_s)^2 \, ds.
\]

Moreover, assuming the evolution of $W_t$ is sufficiently smooth, the driving function $\widetilde{W}_t$ evolves according to
\begin{equation}
\dot{\widetilde{W}}_t = \partial_t \Psi_t(W_t) + \Psi_t'(W_t) \dot{W}_t = -3 \Psi_t''(W_t) + \Psi_t'(W_t) b(W_t),
\end{equation}
where the identity $\partial_t \Psi_t(W_t) = -3 \Psi_t''(W_t)$ follows from Equation (4.35) in \cite{Lawler:book}.
\end{thm}

\begin{proof}
See Section~4.6.1 in~\cite{Lawler:book}.
\end{proof}

\begin{thm}[Stochastic Loewner chain under coordinate change]
\label{random coordinate change}

Assume the setting of Theorem~\ref{Loewner coordinate change in H}, and suppose the driving function \( W_t \in \mathbb{R} \) evolves according to the stochastic differential equation
\begin{equation}
dW_t = \sqrt{\kappa} \, dB_t + b\left(W_t; \Psi_t(W_1), \ldots, \Psi_t(W_n)\right)\, dt,
\end{equation}
where \( B_t \) is standard Brownian motion, and \( b \) is a drift term depending on the conformal images of marked points under the time-dependent conformal map
\[
\Psi_t := \widetilde{g}_t \circ \Psi \circ g_t^{-1},
\]
with \( \Psi \) conformal near \( W_0 \in \mathbb{R} \) and \( g_t \), \( \widetilde{g}_t \) the Loewner flows before and after coordinate change, respectively.

Define the transformed driving function \( \widetilde{W}_t := \Psi_t(W_t) \), and introduce the reparameterized time
\[
s(t) := \int_0^t |\Psi_u'(W_u)|^2 \, du.
\]
Then the time-changed process \( \widetilde{W}_s := \widetilde{W}_{t(s)} \) satisfies the stochastic differential equation
\begin{equation}
\label{chordal-coordinate-changed-sde}
d\widetilde{W}_s = \sqrt{\kappa} \, dB_s 
+ \frac{b\left(W_s; \Psi_{t(s)}(W_1), \ldots, \Psi_{t(s)}(W_n)\right)}{\Psi_{t(s)}'(W_s)} \, ds 
+ \frac{\kappa - 6}{2} \cdot \frac{\Psi_{t(s)}''(W_s)}{[\Psi_{t(s)}'(W_s)]^2} \, ds.
\end{equation}
\end{thm}
\begin{proof}
We apply Itô's formula to the process \( \widetilde{W}_t = \Psi_t(W_t) \), where both \( \Psi_t \) and \( W_t \) are time-dependent. Using the chain rule for semimartingales, we have:
\begin{align*}
d\widetilde{W}_t 
&= (\partial_t \Psi_t)(W_t)\, dt + \Psi_t'(W_t)\, dW_t + \frac{1}{2} \Psi_t''(W_t)\, d\langle W \rangle_t.
\end{align*}
Since \( dW_t = \sqrt{\kappa}\, dB_t + b(W_t; \Psi_t(W_1), \ldots, \Psi_t(W_n))\, dt \), the quadratic variation is \( d\langle W \rangle_t = \kappa\, dt \). Substituting, we get:
\begin{align*}
d\widetilde{W}_t 
&= (\partial_t \Psi_t)(W_t)\, dt 
+ \Psi_t'(W_t) \left[ \sqrt{\kappa}\, dB_t + b(W_t; \Psi_t(W_1), \ldots, \Psi_t(W_n))\, dt \right] 
+ \frac{\kappa}{2} \Psi_t''(W_t)\, dt \\
&= \Psi_t'(W_t)\, \sqrt{\kappa}\, dB_t 
+ \Psi_t'(W_t)\, b(W_t; \Psi_t(W_1), \ldots, \Psi_t(W_n))\, dt 
+ (\partial_t \Psi_t)(W_t)\, dt 
+ \frac{\kappa}{2} \Psi_t''(W_t)\, dt.
\end{align*}

According to Equation (4.35) in \cite{Lawler:book}, we have
\[
(\partial_t \Psi_t)(W_t) = -3 \Psi_t''(W_t).
\]
Substituting this identity, we obtain:
\begin{align*}
d\widetilde{W}_t 
&= \Psi_t'(W_t)\, \sqrt{\kappa}\, dB_t 
+ \Psi_t'(W_t)\, b(W_t; \Psi_t(W_1), \ldots, \Psi_t(W_n))\, dt 
+ \left( \frac{\kappa}{2} - 3 \right) \Psi_t''(W_t)\, dt \\
&= \Psi_t'(W_t)\, \sqrt{\kappa}\, dB_t 
+ \Psi_t'(W_t)\, b(W_t; \Psi_t(W_1), \ldots, \Psi_t(W_n))\, dt 
+ \frac{\kappa - 6}{2} \Psi_t''(W_t)\, dt,
\end{align*}
which completes the derivation of the SDE for \( \widetilde{W}_t \).

To reparameterize time, define
\[
s(t) := \int_0^t |\Psi_u'(W_u)|^2 \, du.
\]
Let \( \widetilde{W}_s := \widetilde{W}_{t(s)} \). By standard time-change theory for semimartingales (e.g., see Revuz–Yor), the transformed Brownian motion is
\[
B_s := \int_0^{t(s)} \Psi_u'(W_u)\, dB_u,
\]
which is again a standard Brownian motion with respect to the time-changed filtration. Dividing all drift and diffusion terms in \( d\widetilde{W}_t \) by \( |\Psi_t'(W_t)| \), we obtain the SDE for \( \widetilde{W}_s \):
\begin{align*}
d\widetilde{W}_s 
&= \sqrt{\kappa}\, dB_s 
+ \frac{b(W_s; \Psi_{t(s)}(W_1), \ldots, \Psi_{t(s)}(W_n))}{\Psi_{t(s)}'(W_s)}\, ds 
+ \frac{\kappa - 6}{2} \cdot \frac{\Psi_{t(s)}''(W_s)}{[\Psi_{t(s)}'(W_s)]^2}\, ds.
\end{align*}
This completes the proof.
\end{proof}

\begin{remark}[Drift term as a pre-Schwarzian form]\label{drift term pre schwarz form}
As a consequence of Theorem~\ref{random coordinate change}, under a conformal change of coordinates \( \tau \), the drift term in the marginal law transforms as a pre-Schwarzian form. That is, if the original drift is \( b \), then the transformed drift \( \widetilde{b} \) satisfies
\[
b = \tau' \cdot \widetilde{b} \circ \tau + \frac{6 - \kappa}{2} \left( \log \tau' \right)'.
\]
\end{remark}

\begin{cor}\label{gamma 1 gamma2 capacity change}
Let $\gamma$, $\tilde{\gamma}$ be two hulls starting at $x \in \partial \mathbb{H}$ and $y \in \partial \mathbb{H}$ with capacity $\epsilon$ and $c \epsilon$ , let $g_{\epsilon}$ be the normalized map removing $\gamma$ and $\tilde{\epsilon}= \operatorname{hcap}(g_{\epsilon}\circ \gamma(t))$, then we have:

\begin{equation}
\tilde{\varepsilon}=
c \varepsilon\left(1-\frac{4\varepsilon}{(x-y)^2}\right)+o\left(\varepsilon^2\right)
\end{equation}

\end{cor}

\begin{proof}
From the Loewner equation, $\partial_t h_t^{\prime}(w)=- \frac{2h_t^{\prime}(w)}{\left(h_t(w)-x_t)\right)^2}$, which implies $h_{\varepsilon}^{\prime}(y)=1-\frac{4\varepsilon}{(x-y)^2}+o(\varepsilon)$. By conformal transformation $h_{\epsilon}(y)$, we get:
$$
\tilde{\varepsilon}=c\epsilon( h'_{\epsilon}(y)^2+ o(\epsilon)) =c \varepsilon\left(1-\frac{4\varepsilon}{(x-y)^2}\right)+o\left(\varepsilon^2\right)
$$

\end{proof}

\subsection{Local commutation relation and null vector equations in $\kappa>0$ case} \label{reparametrization symmetry}

In this section, we explore how the commutation relations (reparametrization symmetry) and conformal invariance impose constraints on the drift terms $b_j(\boldsymbol{z},u)$ in the marginal law for the multiple chordal SLE($\kappa$) system.

The pioneering work on commutation relations was done in \cite{Dub07}. The author studied the commutation relations for multiple SLEs in the upper half plane $\mathbb{H}$ with
$n$ growth points $z_1,z_2,\ldots,z_n \in \mathbb{R}$ and $m$ additional marked points $u_1,u_2,\ldots,u_m \in \mathbb{R}$. We cite his results here.

A significant difference for the general type multiple chordal SLE($\kappa$) system arises when we study their conformal invariance properties. Although the general multiple chordal SLE($\kappa$) systems are conformally invariant, the partition functions in their corresponding equivalence classes do not necessarily exhibit conformal covariance.
However, it is still possible to find at least one conformally covariant partition function within each equivalence class.

The domain Markov property of local multiple SLE($\kappa$) system implies that we can first map out $\gamma_{\left[0, t_i\right]}^{(i)}$ using $g_{t_i}^{(i)}$, then mapping out $g_{t_i}^{(i)}\left(\gamma_{\left[0, t_j\right]}^{(j)}\right)$, or vice versa. The image has the same law regardless of the order in which we map out the curves. This is also known as the commutation relations or reparametrization symmetry, which implies the existence of a partition function.

\begin{thm}[Commutation relations for $u=\infty$] \label{commutation relation for chordal with marked point u infty}
In the upper half plane $\mathbb{H}$, $n$ chordal SLEs start at $x_1,x_2,\ldots,x_n\in \partial{\mathbb{H}}$ with a marked boundary point $u$. 
\begin{itemize}

\item[(i)] Let the infinitesimal diffusion generators be
\begin{equation}
\mathcal{M}_i=\frac{\kappa}{2} \partial_{ii}+b_i(x_1, x_2,\ldots, x_n,u) \partial_i +\sum_{j \neq i} \frac{2}{x_j-x_i}\partial_j 
\end{equation}

where $\partial_i=\partial_{x_i}$. If n SLEs locally commute, then the associated infinitesimal generators satisfy:
\begin{equation}\label{commutation relation of generators}
[\mathcal{M}_i, \mathcal{M}_j]=\frac{4}{(x_i-x_j)^2}(\mathcal{M}_j-\mathcal{M}_i)
\end{equation}
There exists a partition function $\psi(\boldsymbol{x})$ such that the drift term $b_i(\boldsymbol{x})$ is given by:
\begin{equation}
b_i(\boldsymbol{x})= \kappa \partial_i \log \psi
\end{equation}
where $\psi$ satisfies the null vector equations with an undetermined function $h_i(x)$.
\begin{equation}\label{null vector equations in mathbb H infty}
\frac{\kappa}{2} \partial_{ii} \psi+\sum_{j \neq i} \frac{2}{x_j-x_i}\partial_i \psi+ \left[\left(1-\frac{6}{\kappa}\right) \sum_{j \neq i}\frac{1}{(x_j-x_i)^2}+h_i(x_i)\right]\psi=0
\end{equation}

\item[(ii)]
 By analyzing the asymptotic behavior of two adjacent growth points $x_i$ and $x_{i+1}$ ( with no marked points between $x_i$ and $x_{i+1}$ on the real line $\mathbb{R}$), we can further show that the $h_i(x)=h_{i+1}(x)$ in null vector equations (\ref{null vector equations in mathbb H}). As a corollary, if all growth points lie consecutively on the real line $\mathbb{R}$ with no marked points between them, then there exists a common function $h(x)$ such that $h(x)=h_1(x)=\ldots=h_n(x)$.
\end{itemize}
\end{thm}

The discussion on commutation relations above extend to arbitrary \( u \in \mathbb{H} \), as described in the following theorem:

\begin{thm}[Commutation relations for $u\in \mathbb{H}$] \label{commutation relation for chordal with marked point u in H}
In the upper half plane $\mathbb{H}$, $n$ chordal SLEs start at $x_1,x_2,\ldots,x_n\in \partial{\mathbb{H}}$ with a marked boundary point $u$. 
\begin{itemize}

\item[(i)] Let the infinitesimal diffusion generators be
\begin{equation}
\mathcal{M}_i=\frac{\kappa}{2} \partial_{ii}+b_i(x_1, x_2,\ldots, x_n,u) \partial_i +\sum_{j \neq i} \frac{2}{x_j-x_i}\partial_j +\frac{2}{u-x_i} \partial_u
\end{equation}

where $\partial_i=\partial_{x_i}$. If n SLEs locally commute, then the associated infinitesimal generators satisfy:
\begin{equation}
[\mathcal{M}_i, \mathcal{M}_j]=\frac{4}{(x_i-x_j)^2}(\mathcal{M}_j-\mathcal{M}_i)
\end{equation}
There exists a partition function $\psi(\boldsymbol{x},u)$ such that the drift term $b_i(\boldsymbol{x},u)$ is given by:
\begin{equation}
b_i(\boldsymbol{x},u)= \kappa \partial_i \log \psi
\end{equation}
where $\psi$ satisfies the null vector equations with an undetermined function $h_i(x,u)$.
\begin{equation}\label{null vector equations in mathbb H}
\frac{\kappa}{2} \partial_{ii} \psi+\sum_{j \neq i} \frac{2}{x_j-x_i}\partial_i \psi+\frac{2}{u-x_i} \partial_u \psi + \left[\left(1-\frac{6}{\kappa}\right) \sum_{j \neq i}\frac{1}{(x_j-x_i)^2}+h_i(x_i,u)\right]\psi=0
\end{equation}

\item[(ii)]
 By analyzing the asymptotic behavior of two adjacent growth points $x_i$ and $x_{i+1}$ ( with no marked points between $x_i$ and $x_{i+1}$ on the real line $\mathbb{R}$), we can further show that the $h_i(x,u)=h_{i+1}(x,u)$ in null vector equations (\ref{null vector equations in mathbb H}). As a corollary, if all growth points lie consecutively on the real line $\mathbb{R}$ with no marked points between them, then there exists a common function $h(x,u)$ such that $h(x,u)=h_1(x,u)=\ldots=h_n(x,u)$.

\end{itemize}
\end{thm}

\begin{proof}[Proof of theorem (\ref{commutation relation for chordal with marked point u infty}) and theorem (\ref{commutation relation for chordal with marked point u in H})]
See \cite{Dub07}.
\end{proof}

\subsection{Conformal covariance representative of partition functions}

Now, let us focus on how \( {\rm Aut}(\mathbb{H}) \)-invariance imposes constraints on the drift terms and how to choose a conformally covariant partition function within the equivalence class.

\begin{defn}[Basic properties of  \( {\rm Aut}(\mathbb{H}) \)]
We summarize some basic properties of the conformal group \( {\rm Aut}(\mathbb{H}) \), 

\( {\rm Aut}(\mathbb{H}) \) is isomorphic to \( PSL_2(\mathbb{R}) \). By the well-known Iwasawa decomposition of $SL(2,\mathbb{R})$, see \cite{S85}, every element \( \tau \in {\rm Aut}(\mathbb{H}) \) can be written as 
\[
\tau(z) = T_{a,n} \circ \rho_{\theta},
\]
where $$ \rho_\theta(z) = \frac{\cos \theta \cdot z + \sin\theta}{-\sin\theta \cdot z + \cos\theta} $$ and $$ T_{a,n}(z) = \frac{az}{nz+\frac{1}{a}} $$
where $a>0$, $n\in \mathbb{R}$ are two real constants.

Geometrically, this decomposition shows that \( {\rm Aut}(\mathbb{H}) \) is an \( S^1 \)-bundle over \( \mathbb{H} \).

\end{defn}

\begin{thm}[Conformal invariance under ${\rm Aut}(\mathbb{H},\infty)$]
\label{conformal invariance aut H infty}
 For conformal map $\tau \in {\rm Aut}(\mathbb{H},\infty)$, $\tau(z)=az+b$, the drift term
$b_i(\boldsymbol{x},u)$ is a pre-schwarz form, i.e.
$$b_i=\tau^{\prime} \widetilde{b}_i \circ \tau+ \frac{6-\kappa}{2}\left(\log \tau^{\prime}\right)^{\prime}= a \cdot \widetilde{b}_i \circ \tau$$

\begin{itemize}
\item[(i)]
The term $h(x)$ in the null vector equation (\ref{null vector equations in mathbb H}) are identical $0$,
$$h(x) \equiv 0$$
\item[(ii)] There exists a dilatation constant $d$ such that
\begin{equation}
\psi\left(x_1,x_2,\ldots,x_n,\infty\right)=\left(a^{\frac{n(\kappa-6)}{2\kappa}-d}\right)\cdot
\psi(ax_1+b,ax_2+b,\ldots,ax_n+b,\infty)
\end{equation}

\end{itemize}

\end{thm}

\begin{thm}[Conformal invariance under ${\rm Aut}(\mathbb{H})$]
\label{conformal invariance aut H}
 For conformal map $\tau \in {\rm Aut}(\mathbb{H})$, the drift term
$b(\boldsymbol{\theta},u)$ is a pre-schwarz form, i.e.
$$b_i=\tau^{\prime} \widetilde{b}_i \circ \tau+ \frac{6-\kappa}{2}\left(\log \tau^{\prime}\right)^{\prime}$$
\begin{itemize}

\item[(i)]
   
Equivalently, in terms of partition function $\psi$, there exists a smooth function  $F(\tau,u): {\rm Aut}(\mathbb{H})\times \mathbb{H} \rightarrow \mathbb{R}$ such that:
\begin{equation}
   \log(\psi)-\log(\psi \circ \tau)+\frac{\kappa-6}{2\kappa}\sum_i \log(\tau'(z_i)) = F(\tau,u) 
\end{equation}
    and $F$ satisfies the following functional equation:
   \begin{equation} \label{functional equation for F}
      F(\tau_1 \tau_2 ,u)= F(\tau_1,\tau_2(u))+F(\tau_2,u)
   \end{equation}
\item[(ii)]
There exists a rotation constant \( \lambda(\infty) \) such that:
    \[
    F(D_a, \infty) = -d\cdot\log(a).
    \]
    \item[(iii)] Suppose \( F_1(\tau, u) \) and \( F_2(\tau, u) \) correspond to partition functions \( \psi_1 \) and \( \psi_2 \). If their dilatation constants \( d_1 = d_2 \), then there exists a function \( g(u) \) such that:
    \[
   \psi_2 = g(u) \cdot \psi_1.
    \]

    \item[(iv)] For \( \tau \in \mathrm{Aut}(\mathbb{H}) \) and $u \in \overline{\mathbb{H}}$, we define:
    \begin{itemize}
    \item if $\tau(u) \neq \infty, u \neq \infty$
    \[
    F(\tau, u) = d \cdot \log \left(\tau'(u)\right),
    \]
    \item If $\tau(u)=\infty$, $u \neq \infty$ 
    \[
    F(\tau, u)= F(-\frac{1}{\tau},u)
    \]
    \item If $\tau(u)\neq \infty$, $u= \infty$
    \[
    F(\tau, u)= F(-\frac{1}{\tau},0)
    \]
    \item If $\tau(z)=az+b$ and $u=\infty$,
    \[
    F(\tau,\infty) = -d \cdot \log(a)
    \]
        \end{itemize}
    Then $F(\tau,u)$ satisfy equation (\ref{functional equation for F}), with dilatation constant $d$.
\end{itemize}
\end{thm}
\begin{thm}\label{conformal covariant solutions} For a multiple chordal SLE($\kappa$) system with $n$ SLEs starting from $x_1,x_2,\ldots,x_n$ with a marked boundary point $u \in \partial \overline{\mathbb{H}}$.
\begin{itemize}
\item[(i)]
Two partition functions $\widetilde{\psi}$ and $\psi$ are considered \emph{equivalent} if they differ by a multiplicative factor $f(u)$.
\begin{equation}
\tilde{\psi}=f(u)\cdot\psi
\end{equation}
where $f(u)$ is an arbitrary positive real smooth function depending on the marked boudary point $u$.
Under this equivalence, $\tilde{\psi}$ and $\psi$ induce identical multiple chordal SLE($\kappa$) systems.
\item[(ii)]
Within the equivalence class, we can choose $\psi$ to satisfy conformal covariance, that this, under a conformal map $\tau \in {\rm Aut}(\mathbb{H})$, the function $\psi(z_1,z_2,\ldots,z_n,u)$ transforms as
\begin{equation}
\psi(z_1,z_2,\ldots,z_n,u)=\left(\prod_{i=1}^{n}\tau'(z_i)^{\frac{6-\kappa}{2\kappa}}\right)\tau'(u)^{d}\psi\left(\tau(z_1),\tau(z_2),\ldots,\tau(z_n),\tau(u)\right)
\end{equation}
where we understand $\tau'(u)$ as
\begin{itemize}
\item If $\tau(u) \neq \infty, u \neq \infty$,  $\tau'(u)$ is defined as usual,
\item If $\tau(u)=\infty$, $u \neq \infty$ 
    \[
   \tau'(u):= \left(-\frac{1}{\tau(u)}\right)'
    \]
    \item If $\tau(u)\neq \infty$, $u= \infty$
    \[
    \tau'(u):= \left(-\frac{1}{\tau(0)}\right)'
    \]
    \item If $\tau(z)=az+b$ and $u=\infty$,
    \[
    \tau'(u):= a^{-1}
    \]
\end{itemize}
Here, $d$ is an undetermined real constant, representing the scaling exponents at $u$.
\end{itemize}

\end{thm}

\begin{proof}[Proof of theorem (\ref{conformal invariance aut H infty})]
\
\begin{itemize}
    \item[(i)] By the null vector equation (\ref{null vector equations in mathbb H infty}),
\begin{equation}
\begin{aligned}
h_i(x_i)&=-\frac{\kappa}{2} \frac{\partial_{ii} \psi}{\psi}-\sum_{j \neq i} \frac{2}{x_j-x_i}\frac{\partial_i \psi}{\psi}-\left(1-\frac{6}{\kappa}\right) \sum_{j \neq i}\frac{1}{(x_j-x_i)^2}
\\&= \frac{\kappa}{2}b^{2}_i -\frac{\kappa}{2}\partial_i b_i-\frac{2}{x_j-x_i}b_i-\left(1-\frac{6}{\kappa}\right) \sum_{j \neq i}\frac{1}{(x_j-x_i)^2}
\end{aligned}
\end{equation}
Since $b_i$ is translation invariant and homogeneous of degree $-1$ under ${\rm Aut}(\mathbb{H},\infty)$.
By the above equation for $h_i(x_i)$, we obtain that $h_i$ is translation invariant and homogenous of degree $-2$. The only possibility is that
$$ h_i \equiv 0$$
\item[(ii)]

Since $b_i = \kappa\partial_i \log(\psi)$, by the dilatation invariance of $b_i$, for dilatation transformation $D_a$:
$$
\partial_i \left(\log(\psi)- \log(\psi \circ D_a) \right)=0
$$

for $i=1,2,\ldots,n$. Thus, independent of $x_1,x_2,\ldots,x_n$. We obtain that there exists a function $F(a): \mathbb{R} \rightarrow \mathbb{R}$ such that 
$$
   \log(\psi)-\log(\psi \circ D_a)=F(a)
$$
By the chain rule, for $a_1,a_2\in \mathbb{R}^+$,  $F$ satisfies the Cauchy functional equation
$$F(a_1)+F(a_2)=F(a_1a_2)$$
The only solution for the Cauchy functional equation is linear. Thus, there exists a constant $d \in \mathbb{R}$.
$$F(a)=-d \log(a)$$

Similarly, let $L_b$ be the translation transformaiton,

$$
\partial_i \left(\log(\psi)- \log(\psi \circ L_b) \right)=0
$$

for $i=1,2,\ldots,n$. Thus, independent of $x_1,x_2,\ldots,x_n$. We obtain that there exists a function $G(b): \mathbb{R} \rightarrow \mathbb{R}$ such that

$$
   \log(\psi)-\log(\psi \circ L_b)=F(b)
$$
By the chain rule, for $b_1,b_2\in \mathbb{R}$,  $G$ satisfies the Cauchy functional equation
$$G(b_1)+G(b_2)=G(b_1+b_2)$$
The only solution for the Cauchy functional equation is linear. Thus, there exists a constant $d \in \mathbb{R}$.
$$G(a)=-c \cdot b$$

On the other hand note that
$$az+b= a(z+\frac{b}{a})$$
$$D_a \circ L_{\frac{b}{a}} = L_b \circ D_a$$
By the chain rule, 
$$\log \psi \circ D_a \circ L_{\frac{b}{a}}- \log \psi = \log \psi \circ D_a \circ L_{\frac{b}{a}}-\log \psi \circ D_a+ \log \psi \circ D_a- \log \psi =  c \cdot \frac{b}{a}-d \cdot \log b$$
$$\log \psi \circ L_b \circ D_a- \log \psi = \log \psi \circ L_b \circ D_a-\log \psi \circ L_b+ \log \psi \circ L_b- \log \psi = c\cdot b - d\cdot \log b$$
combining above two equations, we obtain that for arbitrary $a>0$ and $b \in \mathbb{R}$:
$$c \cdot b -d \cdot \log b = c \cdot \frac{b}{a}-d \cdot \log b$$
This implies $c =0$.
\end{itemize}
\end{proof}

\begin{proof}[Proof of theorem (\ref{conformal invariance aut H})]
\
\begin{itemize}
\item[(i)] Note that by corollary (\ref{drift term pre schwarz form}), under a conformal map $\tau \in {\rm Aut}(\mathbb{H})$, the drift term $b_i(z_1,z_2,\ldots,z_n,u)$ transforms as

$$
b_i=\tau^{\prime}(z_i) \left(b_i \circ \tau \right)+ \frac{6-\kappa}{2}\left(\log \tau^{\prime}(z_i)\right)^{\prime}
$$
   Since $b_i = \kappa\partial_i \log(\psi)$
$$
\kappa\partial_i \log(\psi)=\kappa\tau^{\prime}(z_i) \partial_i \log(\psi \circ \tau) + \frac{6-\kappa}{2}\left(\log \tau^{\prime}(z_i)\right)^{\prime}
$$
which implies 
$$
\partial_i\left(\log(\psi)-\log(\psi \circ \tau)+\frac{\kappa-6}{2\kappa}\sum_i \log(\tau'(z_i)) \right)=0
$$
for $i=1,2,\ldots,n$. Thus, independent of variables $z_1,z_2,\ldots,z_n$. 

We obtain that there exists a function $F: {\rm Aut}(\mathbb{H})\times \mathbb{H} \rightarrow \mathbb{C}$ such that 
$$
   \log(\psi)-\log(\psi \circ \tau)+\frac{\kappa-6}{2\kappa}\sum_i \log(\tau'(z_i)) = F(\tau,u) 
$$
By direct computation, we can show that
$$
\begin{aligned}
F(\tau_1 \tau_2 ,u)&= \log(\psi)-\log(\psi \circ \tau_1 \tau_2)+\frac{\kappa-6}{2\kappa}\sum_i \log((\tau_1 \tau_2)'(z_i))  \\
&=\log(\psi)-\log(\psi\circ \tau_2)+\log(\psi\circ \tau_2)-\log(\psi \circ \tau_1 \tau_2) \\
&+\frac{\kappa-6}{2\kappa}\sum_i \log( \tau_{2}'(z_i))+\frac{\kappa-6}{2\kappa}\sum_i \log(\tau_1'(\tau_2(z_i)))\\
&=F(\tau_1,\tau_2(u))+F(\tau_2,u)
\end{aligned}
$$

\item[Part (ii)]

By the functional equation (\ref{functional equation for F}) and $u=\infty$ is the fixed point of the dilatation transformation $D_{a}(z)=az$ and translation transformation $L_{b}(z)=z+b$, we obtain that
$$F(L_{b}\circ D_{a},\infty)=F(L_b,D_a(\infty))+F(D_a,\infty)=F(L_b,\infty)+F(D_a,\infty)$$

Since $F(L_{x+y},\infty)=F(L_x,\infty)+F(L_y,\infty)$, $F(D_{xy},\infty)=F(D_x,\infty)+F(D_y,\infty)$.
These two are Cauchy functional equations, the only solutions are linear, therefore,
there exists two constant $c_1,c_2$ such that:
$F(L_x,\infty)=c_1 x$
$F(D_x,\infty)=c_2 \log x$

On the other hand note that
$$az+b= a(z+\frac{b}{a})$$
$$D_a \circ L_{\frac{b}{a}} = L_b \circ D_a$$
which implies that
$$F(L_{b}\circ D_{a},\infty) = F(D_a \circ L_{\frac{b}{a}},\infty)$$
$$c_1 b + c_2 \log b = c_1 \cdot \frac{b}{a}+ c_2 \log b$$
This implies $c_1 =0$, denote $-c_2 =d$, we obtain the desired result

\item[(iii)]

Let $v = \tau(u)$, denote $S_{u}$ the unique rotation map in the form of 
$$ S_u(z) = \frac{\cos \theta \cdot z + \sin\theta}{-\sin\theta \cdot z + \cos\theta} $$
such that $S_{u}(\infty)= u$

then by the functional equation (\ref{functional equation for F}), we obtain that:
$$
F_i(\tau, u)= F_i(S_v \circ L_{b}\circ D_{a} \circ S^{-1}_{u}, S_u(\infty))
= -F_i(S_u,\infty) +F_i(S_v \circ A_\theta, \infty)
= F_i(S_v,\infty)-F_i(S_u,\infty)- d_i \log(a)
$$
for $i=1,2$.
we define $$f(u) = F_1(S_u,\infty)-F_2(S_u,\infty)$$

Now, suppose $\psi_i$ are corresponding partition functions.
By the definition of function $F(\tau,u)$, $\psi_i$ satisfies the following functional equation
\begin{equation}
   \log(\psi_i)-\log(\psi_i \circ \tau)+\frac{\kappa-6}{2\kappa}\sum_j \log(\tau'(z_j)) = F_i(\tau,u) 
\end{equation} 
Subtracting two equations, we obtain that
$$
\log(\frac{\psi_1}{\psi_2}) - \log( \frac{\psi_1 \circ \tau}{\psi_2 \circ \tau}) = f(v)-f(u)+ (\omega_1-\omega_2) \theta
$$
Then if $d_1=d_2$
$$
\log(\frac{\psi_1}{\psi_2}) - \log( \frac{\psi_1 \circ \tau}{\psi_2 \circ \tau}) = f(v)-f(u)
$$
which is equivalent to
$$ \psi_2 =c e^{f(u)} \psi_1 $$
thus $$g(u) = c e^{f(u)} $$ 
where $c>0$.

\item[(iv)] The identity (\ref{functional equation for F}) can be verified by direct computation.
\end{itemize}

\end{proof}

\begin{proof}[Proof of theorem (\ref{conformal covariant solutions})]
\

\begin{itemize}
 \item[(i)]
The drift term in the marginal law for multiple chordal SLE($\kappa$) systems is given by 
$$b_i =\kappa \partial_j \log(\psi)$$

If two partition functions differ by 
a multiplicative function $f(u)$.
\begin{equation}
\tilde{\psi}=f(u)\cdot\psi
\end{equation}
where $f(u)$ is an arbitrary positive real smooth function depending on the marked interior point $u$.
Note that
$$b_i =\kappa \partial_j \log(\psi)=\kappa \partial_j \log(\widetilde{\psi})= \widetilde{b_i}$$
Thus $\widetilde{\psi}$ and $\psi$ induce identical multiple chordal SLE($\kappa$) system.
\item[(ii)]
For a multiple chordal SLE($\kappa$) system with partition function \(\psi(\boldsymbol{x}, u)\), we proceed as follows:

 Let $d$ be the corresponding dilatation constant. Let $S_u$ be the unique rotation map such that $S_u(\infty)=u$.
 Define:
    \[
    \psi(x_1, x_2, \ldots, x_n, \infty) = \left(\prod_{i=1}^{n} S_u'(x_i)^{\frac{\kappa-6}{2\kappa}}\right)\cdot S_u'(\infty)^{d} \cdot\tilde{\psi}\left(S_u(x_1), S_u(x_2), \ldots, S_u(x_n), u\right)
    \]
    where $S_{u}'(\infty):= \left(-\frac{1}{S_{u}(0)}\right)'$.
    
    Since \(\tilde{\psi}\) and \(\psi\) share the same dilatation constant \(d\). By (iii) of Theorem (\ref{conformal invariance aut H}), there exists a function \(f(u)\) such that:
    \[
    \tilde{\psi} = f(u) \cdot \psi.
    \]
\end{itemize}

\end{proof}

\section{Coulomb gas integral solutions of type $(n,m)$}

\subsection{Classification and link pattern}
\label{Classification of screening solutions}

Based on the Coulomb gas integral method as introduced in \cites{JZ25s,JZ25t}, we are able to construct solutions to the null vector PDEs and Ward's identities via screening.
These solutions satisfy the following null vector equations:

\begin{equation} \label{null vector equation for Screening solutions}
\left[\frac{\kappa}{4} \partial_j^2+\sum_{k \neq j}^{n}\left(\frac{\partial_k}{x_k-x_j}-\frac{(6-\kappa) /  2\kappa}{\left(x_k-x_j\right)^2}\right)+\frac{\partial_{n+1}}{u-x_j} \right.
\left. -\frac{\lambda_{(b)}(u)}{\left(u-x_j\right)^2}\right] \mathcal{J}\left(\boldsymbol{x},u\right)=0
\end{equation}
for $j=1,2,\ldots,n$,
and the following ward identities: 
\begin{equation} \label{Ward identities for screening solutions}
\begin{aligned}
&\left[\sum_{i=1}^{n} \partial_{x_i}+ \partial_u\right] \mathcal{J}(\boldsymbol{x}) =0,\\
& \left[\sum_{i=1}^{n}\left(x_i \partial_{x_i}+\frac{6-\kappa}{2\kappa}\right)+ u\partial_u+\lambda_{(b)}(u)u\right]\mathcal{J}(\boldsymbol{x})=0, \\
& \left[\sum_{i=1}^{n}\left(x_i^2 \partial_{x_i}+\frac{6-\kappa}{\kappa}x_i\right)+ u^2\partial_u+2\lambda_{(b)}(u)u\right] \mathcal{J}(\boldsymbol{x})=0
\end{aligned}
\end{equation}
where $\lambda_{(b)}(u)$ is the conformal dimensions of $u$.

We need to choose a set of integration contours to integrate $\Phi$. we will explain how we choose integration contours, which lead the screening solutions, see theorem (\ref{solution space to null and ward}). We conjecture that these screening solutions span the solution space of the null vector equations (\ref{null vector equation for Screening solutions}) 
and the Ward's identities (\ref{Ward identities for screening solutions}).

To do this, let's begin by defining the link patterns that characterize the topology of integration contours.

\begin{defn}[Chordal link pattern] Given $\boldsymbol{x}=\{x_1,x_2,...,x_n\}$ on the real line, a link pattern is a homotopically equivalent class of non-intersecting curves connecting pair of boundary points (links) or connecting boundary points and the infinity (rays). The link patterns with $n$ boundary points and $m$ links are called $(n,m)$-links, denoted by ${\rm LP}(n,m)$.

The number of chordal $(n,m)$-links is given by $|{\rm LP}(n,m)|=C_{n}^{m+1}-C_{n}^{m}$.
\end{defn}

When all $\sigma_i= a$, $1\leq i \leq n$, we can assign charge $-2a$ or $2(a+b)$ to screening charges.

\begin{itemize}

\item (Chordal ground solutions)
In the upper half plane $\mathbb{H}$, we assign charge $a$ to $x_1,x_2,\ldots,x_n$, charge $-2a$ to $\xi_1,\ldots,\xi_m$ and charge $\sigma_u=2b-(n-2m)a$ to marked points $u$ 
to maintain neutrality condition ($\rm{NC_b}$).
\begin{equation} \label{multiple chordal SLE kappa master function}
\begin{aligned}
\Phi_{\kappa}\left(x_1, \ldots, x_{n}, \xi_1,\xi_2,\ldots,\xi_m, u \right)= & \prod_{i<j}\left(x_i-x_j\right)^{a^2} \prod_{j<k}\left(x_j-\xi_k\right)^{-2a^2}  \prod_{j<k}\left(\xi_j-\xi_k\right)^{4a^2}  \\
& \prod_{j}(x_i-u)^{a(2b-(n-2m)a)}
\prod_{j}(\xi_j-u)^{-2a(2b-(n-2m)a)}
\end{aligned}
\end{equation}

\begin{itemize}
    \item[(1)] $(-2a)\cdot a=-\frac{4}{\kappa}$.  $\xi_i=x_j$ is a singular point of the type $\left(\xi_i-x_j\right)^{-4 / \kappa}$.
     \item[(2)] $(-2a)\cdot (-2a)=\frac{8}{\kappa}$. $\xi_i=\xi_j$ is a singular point of of the type $(\xi_i-\xi_j)^{\frac{8}{\kappa}}$
     \item[(3)] $(-2a)\cdot (2b-(n-2m)a)=\frac{4(n-2m+2)}{\kappa}$.  $\xi_i=u$ is singular point of the type $(\xi_i-u)^{\frac{4(n-2m+2)}{\kappa}}$.
\end{itemize}

In this case, for $m \leq \frac{n+2}{2}$ and a $(n,m)$ chordal link pattern $\alpha$, we can choose $p$ non-intersecting Pochhammer contours $\mathcal{C}_1,\mathcal{C}_2,\ldots,\mathcal{C}_m$ surrounding pairs of points (which correspond to links in a chordal link pattern) to integrate $\Phi_{\kappa}$, we obtain
\begin{equation}
    \mathcal{J}^{(m, n)}_{\alpha}(\boldsymbol{x}):=\oint_{\mathcal{C}_1} \ldots \oint_{\mathcal{C}_m} \Phi_\kappa(\boldsymbol{x}, \boldsymbol{\xi}) d \xi_m \ldots d \xi_1 .
\end{equation}

Note that the charge at $u$ is given by $\sigma_u=2b-(n-2m)a$, thus $$\lambda_{(b)}(u)= \frac{(2m-n)^2}{\kappa}-\frac{2(2m-n)}{\kappa}+\frac{2m-n}{2}$$

The chordal ground solution $\mathcal{J}^{(m,n)}_{\alpha}$ satisfies the null vector equations (\ref{null vector equation for Screening solutions}) and Ward's identities (\ref{Ward identities for screening solutions}) with above $\lambda_{(b)}(u)$.

The number of solutions we can construct via screening is precisely the number of chordal link patterns. We conjecture that these are exactly all the solutions to the null vector equations.

\item (Chordal excited solutions)In the upper half plane $\mathbb{H}$, we assign charge $a$ to $x_1,x_2,\ldots,x_n$, charge $-2a$ to $\xi_1,\ldots,\xi_m$ and charge $2(a+b)$ to $\xi_1,\ldots,\zeta_q$.
Then, we assign charge $\sigma_u=2b-(n-2m)a-2q(a+b)$ to marked points $u$ 
to maintain neutrality condition ($\rm{NC_b}$).
\begin{equation}
\begin{aligned}
&\Phi_{\kappa}\left(x_1, \ldots, x_{n}, \xi_1,\xi_2,\ldots,\xi_m,\zeta_1,\zeta_2,\ldots,\zeta_q, u \right)=  \\
&\prod_{i<j}\left(x_i-x_j\right)^{a^2} \prod_{j<k}\left(x_j-\xi_k\right)^{-2a^2}  \prod_{j<k}\left(\xi_j-\xi_k\right)^{4a^2}  \\
& \prod_{j<k}\left(x_j-\zeta_k\right)^{2a(a+b)}  \prod_{j<k}\left(\zeta_j-\zeta_k\right)^{4(a+b)^2} 
\\
& \prod_{j}(x_i-u)^{a\sigma_u}
\prod_{j}(\xi_j-u)^{-2a\sigma_u}
\prod_{j}(\zeta_j-u)^{2(a+b)\sigma_u}
\end{aligned}
\end{equation}

In the unit disk $\mathbb{H}$, if we set $u=\infty$, then we have

\begin{equation}
\begin{aligned}
&\Phi_{\kappa}\left(x_1, \ldots, x_{n}, \xi_1,\xi_2,\ldots,\xi_m,\zeta_1,\zeta_2,\ldots,\zeta_q\right)=  \\
&\prod_{i<j}\left(x_i-x_j\right)^{a^2} \prod_{j<k}\left(x_j-\xi_k\right)^{-2a^2}  \prod_{j<k}\left(\xi_j-\xi_k\right)^{4a^2}  \\
& \prod_{j<k}\left(x_j-\zeta_k\right)^{2a(a+b)}  \prod_{j<k}\left(\zeta_j-\zeta_k\right)^{4(a+b)^2} 
\end{aligned}
\end{equation}

\begin{itemize}
    \item[(1)] $(-2a)\cdot a=-\frac{4}{\kappa}$.  $\xi_i=z_j$ is a singular point of the type $\left(\xi_i-z_j\right)^{-4 / \kappa}$.
     \item[(2)] $(-2a)\cdot (-2a)=\frac{8}{\kappa}$. $\xi_i=\xi_j$ is a singular point of of the type $(\xi_i-\xi_j)^{\frac{8}{\kappa}}$
     \item[(3)] $(-2a)\cdot (b-\frac{(n-2m)a}{2}-q(a+b))=\frac{2(n-2m+2)}{\kappa}+q$.  $\xi=u$ and $\xi=u^*$ are singular points of the type $(\xi_i-u)^{\frac{2(n-2m+2)}{\kappa}+q}$ and $(\xi_i-u^*)^{\frac{2(n-2m+2)}{\kappa}+q}$
     \item[(4)] $2(a+b)\cdot (2b-(n-2m)a-2q(a+b))=\frac{(1-q)\kappa}{2}-n+2m-2$.  $\xi_i=u$ is a singular point of the type $(\xi_i-u)^{\frac{(1-q)\kappa}{2}-n+2m-2}$.
\end{itemize}

For $q=1$, $\zeta_1=u$ is one singular point of degree $-n+2m-2$.
We have only one choice for screening contours to integrate $\zeta_1$, the circle $C(u,\epsilon)$ around $u$ with radius $\epsilon$, this gives the excited solution.
    
In this case, for $m \leq \frac{n+2}{2}$ and a $(n,m)$ chordal link pattern $\alpha$, we can choose $p$ non-intersecting Pochhammer contours $\mathcal{C}_1,\mathcal{C}_2,\ldots,\mathcal{C}_m$ surrounding pairs of points (which correspond to links in a chordal link pattern) to integrate $\Phi_{\kappa}$, we obtain
\begin{equation}
    \mathcal{K}^{(m,n)}_{\alpha}(\boldsymbol{z}):=\oint_{\mathcal{C}_1} \ldots \oint_{\mathcal{C}_m}\oint_{C(u,\epsilon)} \Phi_\kappa(\boldsymbol{z}, \boldsymbol{\xi},\zeta_1) d \xi_m \ldots d \xi_1 d\zeta_1.
\end{equation}
\end{itemize} 

Note that the charges at $u$ is given by $\sigma_u=(2m-n-2)a$
$$\lambda_{(b)}(u)= \frac{(2m-n)(2m-n-2)}{\kappa}-\frac{2m-n-2}{2}$$ 
The chordal excite solution $\mathcal{K}^{(m,n)}_{\alpha}$ 
satisfies the null vector equations (\ref{null vector equation for Screening solutions}) and Ward's identities (\ref{Ward identities for screening solutions}) with above $\lambda_{(b)}(u)$.

For $q\geq 2$, since $u$ is the only singular points for screening charges,
it is impossible to choose two non-intersecting contours for $\{\zeta_1,\zeta_2,\ldots,\zeta_q \}$.

We conjecture that the set of screening solutions constructed via Coulomb gas integrals indexed by chordal link patterns forms a basis for the space of solutions to the null vector equations and Ward identities in the presence of a marked boundary point.

\section{Relations to Calogero-Moser system}

\subsection{Null vector equations and quantum Calogero-Moser system}

In this section, we establish parallel connections between multiple chordal SLE$(\kappa)$ systems and the quantum Calogero--Moser system. Specifically, we show that any partition function satisfying the null vector equations corresponds to an eigenfunction of the quantum Calogero--Moser Hamiltonian, as first discovered in~\cite{Car04}.

\begin{proof}[Proof of theorem (\ref{CM results kappa>0})]
\
\begin{itemize}

\item[(i)] Recall that the null vector differential operator $\mathcal{L}_j$ is given by

\begin{equation}
\begin{aligned}
\mathcal{L}_j= & \frac{\kappa}{2}\left(\frac{\partial}{\partial x_j}\right)^2 +\sum_{k \neq j}\left( \frac{2}{x_k-x_j} \frac{\partial}{\partial x_k}+\left(1-\frac{6}{\kappa}\right)\frac{1}{(x_k-x_j)^2}\right) .
\end{aligned}    
\end{equation}

Then, the null vector equations for $\psi(\boldsymbol{x})$ can be written as

\begin{equation}
    \mathcal{L}_j \psi(\boldsymbol{x})= 0
\end{equation}
for $j=1,2,\ldots,n$.

To simplify the formula, we introduce the notation,
$$
f(x)= \frac{2}{x}, \quad f_{j k}=f\left(x_j-x_k\right), \quad F_j=\sum_{k \neq j} f_{j k} .
$$

$$
f^{\prime}(x)=-\frac{1}{2x^2}, \quad f^{\prime}_{j k}=f^{\prime}\left(x_j-x_k\right), \quad F^{\prime}_j=\sum_{k \neq j} f^{\prime}_{j k} 
$$

Using this notation, we have
$$
\mathcal{L}_j=\frac{\kappa}{2} \partial_j^2+\sum_{k \neq j}f_{kj} \partial_k+\sum_{k \neq j}(1-\frac{6}{\kappa}) f_{j k}^{\prime}
$$

with $\partial_j = \frac{\partial}{\partial x_j}$ and the Calogero-Moser hamiltonian can be written as
\begin{equation}
H_n(\beta)=-\sum_j\left(\frac{1}{2} \partial_j^2+\frac{\beta(\beta-2)}{16} F_j^{\prime}\right) \text {. }
\end{equation}
where $\beta=\frac{8}{\kappa}$. 
\\ \indent
To relate the null-vector equations to the Calogero-Moser system, we sum up the null-vector operators. Let

\begin{equation}
\mathcal{L}=\sum_j \mathcal{L}_j=\frac{\kappa}{2} \sum_j \partial_j^2+\sum_j\left(F_j \partial_j+h F_j^{\prime}\right)    
\end{equation}

Then the partition functions $\psi(\boldsymbol{x})$ are eigenfunctions of $\mathcal{L}$ with eigenvalue $0$. 
\begin{equation}
\mathcal{L} \psi(\boldsymbol{x})=0
\end{equation}

Recall that $$\Phi_{r}(\boldsymbol{x})=\prod_{1 \leq j<k \leq n}\left(x_j-x_k\right)^{-2r}$$

From the properties $\partial_j \Phi_r=-r \Phi_r F_j$ and $\sum_j F_j^2=-2 \sum_j F_j^{\prime}$, we can check that
$$
\Phi_{-\frac{1}{\kappa}} \cdot \mathcal{L} \cdot \Phi_{\frac{1}{\kappa}}=\kappa H_n\left(\frac{8}{\kappa}\right)
$$
which implies
$$
\tilde{\psi}(\boldsymbol{x})=\Phi_{\frac{1}{\kappa}}^{-1}(\boldsymbol{x})\psi(\boldsymbol{x})
$$
is an eigenfunction of the Calogero-Moser hamiltonian $H_n\left(\frac{8}{\kappa}\right)$, with eigenvalue
\begin{equation}
E=0.
\end{equation}  
\item[(ii)] see Equation (\ref{commutation relation of generators}) in Theorem (\ref{commutation relation for chordal with marked point u infty}).
\end{itemize}
\end{proof}

\section*{Acknowledgement}
I express my sincere gratitude to Professor Nikolai Makarov for his invaluable guidance. I am also thankful to Professor Eveliina Peltola and Professor Hao Wu for their enlighting discussions.

\end{document}